\documentclass[10pt,twoside]{amsart} 
\usepackage{amsmath, amsthm, amscd, amsfonts, amssymb, graphicx, color}
\usepackage[bookmarksnumbered, colorlinks, plainpages]{hyperref}

\newtheorem{theorem}{Theorem}[section]
\newtheorem{lemma}[theorem]{Lemma}
\newtheorem{proposition}[theorem]{Proposition}
\newtheorem{corollary}[theorem]{Corollary}

\theoremstyle{definition}
\newtheorem{definition}[theorem]{Definition}
\newtheorem{example}[theorem]{Example}

\newtheorem{remark}[theorem]{Remark}

\numberwithin{equation}{section}

\def\<{\langle}
\def\>{\rangle}

\long\def\alert#1{\smallskip{\hskip\parindent\vrule%
\vbox{\advance\hsize-2\parindent\hrule\smallskip\parindent.4\parindent%
\narrower\noindent#1\smallskip\hrule}\vrule\hfill}\smallskip}

\begin{document}

\title[\ci\ On Endo-Prime and Endo-Coprime Modules]{On Endo-Prime and Endo-Coprime Modules}

\author{Hojjat Mostafanasab and Ahmad Yousefian Darani }
\address{Department of Mathematics and Applications, University of Mohaghegh Ardabili, P. O. Box 179, Ardabil, Iran}
\email{h.mostafanasab@uma.ac.ir, h.mostafanasab@gmail.com}

\address{Department of Mathematics and Applications, University of Mohaghegh Ardabili, P. O. Box 179, Ardabil, Iran}
 \email{yousefian@uma.ac.ir, youseffian@gmail.com}




\begin{abstract}
The aim of this paper is to investigate properties of endo-prime
and endo-coprime modules which are generalizations of prime and
simple rings, respectively. Various properties of endo-coprime
modules are obtained. Duality-like connections are established for
endo-prime and endo-coprime modules.
\end{abstract} 

\maketitle

\section{Introduction and preliminaries}

\noindent Throughout all rings are associative with non-zero
identity elements and modules are unital. We know that $R$ is a
prime ring if and only if any non-zero ideal of $R$ has zero left
annihilator, or stated otherwise, any non-zero fully invariant
submodule of $_{R}R$ is faithful over the endomorphism ring ${\rm
End}_{R}(R)\simeq R$. Haghany and Vedadi \cite{hag} generalized
this property to modules: An $R$-module $M$ with $S={\rm
End}_{R}(M)$ is called {\it endo-prime} if for any non-zero fully
invariant submodule $K$ of $M$, ${\rm Ann}_{S}(K)$=0. They show
that being endo-prime is a Morita invariant property, and an
endo-prime module has a prime endomorphism ring. An $R$-module $M$
is a direct sum of isomorphic simple modules if and only if each
non-zero element of $\sigma[M]$ is an endo-prime module. The dual
notion of endo-prime modules defined by Wijayanti \cite{wij}: An
$R$-module $M$ with $S={\rm End}_{R}(M)$ is called {\it
endo-coprime} if for any proper fully invariant submodule $K$ of
$M$, ${\rm Ann}_{S}(M/K)$=0. In the special case, $_{R}R$ is
endo-coprime if and only if $R$ is a simple ring. It is shown in
\cite{wij} that endo-coprime modules have prime endomorphism
rings.

In section 2, we give sufficient conditions for an endo-prime
module to be fully prime and polyform. Some general properties of
endo-coprime $R$-module $M$ are obtained. We see that for
endo-coprime $fi$-coretractable module $_{R}M$, ${\rm Ann}_{R}(M)$
is a prime ideal of $R$. Now consider comodules over coalgebra $C$
over a commutative ring $R$, provided that the coalgebra $C$
satisfies the $\alpha$-condition. In this context one of the
questions one may ask is when the dual algebra $C^{*}={\rm
Hom}_{R}(C,R)$ of an $R$-coalgebra $C$ is a prime algebra. The
first paper to consider this was by Xu, Lu, and Zhu \cite{xu} who
observed that this is the case if $C$ is a coalgebra over a field
$k$ and $(C^{*}*f)\rightharpoonup C=C$ for any non-zero element
$f\in C^{*}$. Another approach in this direction can be found in
Jara et. al. \cite{jar} and Nekooei-Torkzadeh \cite{nek} where
coprime coalgebras (over fields) are defined by using the wedge
product. We show that a coalgebra $C$ over a base field is fully
coprime if and only if its dual algebra $C^{*}$ is prime. It turns
out that each non-zero element of $\sigma[M]$ is an endo-prime
module if and only if each non-zero element of $\sigma[M]$ is an
endo-coprime module. Being endo-coprime is also a Morita invariant
property.

In section 3, we shall deal with $R$-duals. It is proved that for
a finitely generated module $M$ over a quasi-Frobenius ring $R$,
$M$ is an endo-prime (resp. endo-coprime) $R$-module if and only
if $M^{*}={\rm Hom}_{R}(M,R)$ is an endo-coprime (resp.
endo-prime) $R$-module.

As before, $_{R}M$ is a non-zero left module over the ring $R$,
its endomorphism ring ${\rm End}_{R}(M)$ will act on the right
side of $_{R}M$, in other words, $_{R}M_{{\rm End}_{R}(M)}$ will
be studied mainly. For the convenience of the readers, some
definitions of modules that will be used in the next sections are
provided. Let $M$ be a left $R$-module. We say that $N\in R$-Mod
is {\it subgenerated} by $M$ if $N$ is a submodule of an
$M$-generated module (see the \cite{wis1}). The category of
$M$-subgenerated modules is denoted by $\sigma[M]$. For submodule
$N$ of $M$, we write $N\unlhd M$ when $N$ is an essential
submodule, and $N\ll M$ when $N$ is a superfluous (or small)
submodule of $M$. In the category of left $R$-modules there are
various notions of prime objects which generalize the well-known
notion of a prime associative (commutative) ring $R$. For the
notions of (co)primeness of modules we refer to \cite{wij} and
\cite{wis2}.

\begin{definition}
{\rm Recall that $_{R}M$ is\\
--{\it prime} if for any non-zero (fully invariant) submodule $K$
of $M$, ${\rm Ann}_{R}(K)={\rm Ann}_{R}(M)$.\\
--{\it coprime} if for any proper (fully invariant) submodule $K$
of $M$, ${\rm Ann}_{R}(M/K)={\rm Ann}_{R}(M)$.\\
--{\it fully prime} if for any non-zero fully invariant submodule
$K$ of $M$, $M$ is $K$-cogenerated.\\
--{\it fully coprime} if for any proper fully invariant
submodule $K$ of $M$, $M$ is $M/K$-generated.\\
--{\it strongly coprime} if for any proper fully invariant
submodule $K$ of $M$, $M$ is subgenerated by $M/K$, i.e., $M\in
\sigma[M/K]$.}
\end{definition}

$M$ is called {\it retractable} (resp. $fi$-{\it retractable}) if
for any non-zero submodule (resp. fully invariant submodule) $K$
of $M$, ${\rm Hom}_{R}(M,K)\neq0.$\\
Dually, $M$ is called  {\it coretractable} (resp. $fi$-{\it
coretractable}) if for any proper submodule (resp. fully invariant
submodule) $K$ of $M$, $$\pi_{K}\diamond{\rm
Hom}_{R}(M/K,M)=\{f\in S\mid(K)f=0\}\neq0.$$ We use $\diamond$ for
the composition of mappings written on the right side. The usual
composition is denoted by $\circ$. Thus, from now on, we use
$(u)f\diamond g=g\circ f(u)$.

\section{Endo-Prime and Endo-Coprime Modules}

We begin with investigating the relation between endo-prime and
fully prime modules. For any fully invariant submodules $K,L$ of
$M$, consider the product
$$K\ast_{M}L:=K{\rm Hom}_{R}(M,L).$$ According to
{\rm\cite[1.6.3]{wij}}, $M$ is fully prime if and only if for any
fully invariant submodules $K,L$ of $M$, the relation
$K\ast_{M}L=0$ implies $K=0$ or $L=0$.

\begin{lemma}\label{fp}
Let $_{R}M$ be a $fi$-retractable module and $S={End}_{R}(M)$. If
$S$ is prime, then $M$ is fully prime.
\end{lemma}

\begin{proof}
Let $K,L$ be non-zero fully invariant submodules of $M$ which
satisfy $K\ast_{M}L=K{\rm Hom}_{R}(M,L)=0$. By assumption we have
$$M{\rm Hom}_{R}(M,K){\rm Hom}_{R}(M,L)=0.$$ Since $S$ is prime,
${\rm Hom}_{R}(M,K)=0$ or ${\rm Hom}_{R}(M,L)=0$, a contradiction,
because $M$ is $fi$-retractable. Consequently $M$ is fully prime.
\end{proof}

A submodule $U$ of $R$-module $N$ is called $M$-{\it rational} in
$N$ if for any $U\subseteq V\subseteq N$, ${\rm
Hom}_{R}(V/U,M)=0$. $M$ is called {\it polyform} if any essential
submodule is rational in $M$. The dual notions are: A submodule
$X$ of $N$ is called $M$-{\it corational} in $N$ if for any
$Y\subseteq X\subseteq N$, ${\rm Hom}_{R}(M,X/Y)=0$. $M$ is called
{\it copolyform} if any superfluous submodule is corational
in $M$.\\

An $R$-module $E$ is called {\it pseudo-injective} in $\sigma[M]$
if any diagram in $\sigma[M]$ with exact row

$$\begin{CD}
0 @>>> L @>f>> N\\
@. @VgVV \\
@. E
\end{CD}$$

can be extended nontrivially by some $s\in{\rm End}_{R}(E)$ and
$h:N\rightarrow E$ to the commutative diagram

$$\begin{CD}
0 @>>> L @>f>> N\\
@. @VgVV @VVhV\\
@. E @>>s> E,
\end{CD}$$

that is, $gs=fh\neq0$ (see \cite{cla}).\\

The following result shows sufficient conditions for endo-prime
modules to be fully prime and polyform.
\begin{proposition}\label{pseudo}
Let $M$ be an endo-prime $R$-module with $S=End_{R}(M)$. Then the following statements hold:\\
$(1)$ If $M$ is $fi$-retractable, then $M$ is fully prime.\\
$(2)$ If $M$ is semi-injective, then the center of $S$ is a field.\\
$(3)$ If $M$ is pseudo-injective in $\sigma[M]$ with
$Soc(M)\neq0$, then $M$ is polyform.
\end{proposition}

\begin{proof}
$(1)$ Since $_{R}M$ is endo-prime, then $S$ is prime.
Hence the assertion follows from Lemma \ref{fp}.\\
$(2)$ Let $f$ be a non-zero central element of $S$. Then ${\rm
Ker}~f$ is a fully invariant submodule of $M$. Because $M$ is
semi-injective, $fS={\rm Ann}_{S}({\rm Ker}~f)$. Now since $M$ is
endo-prime, we must have ${\rm Ker}~f=0$ which implies that
$fS=S$. Consequently $f$ is an invertible element of $S$.\\
$(3)$ Because $M$ is endo-prime, we see that for any non-zero
fully invariant submodule $K$ of $M$, ${\rm Hom}_{R}(M/K,M)=0$.
Since ${\rm Soc}(M)$ is a non-zero fully invariant submodule of
$M$, ${\rm Hom}_{R}(M/{\rm Soc}(M),M)=0$. For any essential
submodule $L\unlhd M$, ${\rm Soc}(M)\subseteq L$. Thus ${\rm
Hom}_{R}(M/L,M)=0$ and by pseudo-injectivity of $M$ in
$\sigma[M]$, ${\rm Hom}_{R}(L'/L,M)=0$ for any $L\subseteq L'$,
i.e., $L$ is $M$-rational.
\end{proof}

\begin{corollary}
The following statements hold on a prime ring $R$:\\
$(1)$ If $_{R}R$ is semi-injective, then the center of $R$ is a field.\\
$(2)$ If $_{R}R$ is pseudo-injective with $Soc(R)\neq0$, then
$_{R}R$ is polyform.
\end{corollary}

In continuation, we study endo-coprime modules. The following
examples show that the concepts endo-coprime and coprime are
different in general.
\begin{example}{\rm
$(1)$ Let $F$ be a field, $R=\left(\begin{array}{cccc}F & F\\0 &
F\end{array}\right)$, $e=\left(\begin{array}{cccc}1 & 0\\0 &
0\end{array}\right)$, $M=Re$. Then ${\rm End}_{R}(M)\simeq
eRe\simeq F$ as rings, so $_{R}M$ is endo-coprime by
{\cite[1.5.2 part (1)]{wij}}. Now consider two left ideals
$I=\left(\begin{array}{cccc}0 & F\\0 & 0\end{array}\right)$ and
$J=\left(\begin{array}{cccc}0 & F\\0 & F\end{array}\right)$ of
$R$. We can easily see that $I,J\subseteq{\rm Ann}_{R}(M)$ and
$IJ=0$. Thus ${\rm Ann}_{R}(M)$ is not a prime ideal of $R$. Then
$_{R}M$ is not coprime, by {\rm\cite[Lemma 1.3.3 part (1)]{wij}}.\\
$(2)$ Notice that ${\rm
End}_{\mathbb{Z}}(\mathbb{Q})\simeq\mathbb{Q}$ and hence
$\mathbb{Q}$ has no nontrivial fully invariant submodules. Then
$\mathbb{Q}$ is a coprime $\mathbb{Z}$-module. Also any non-zero
factor module $\mathbb{Z}_{p^{\infty}}/K$ of
$\mathbb{Z}_{p^{\infty}}$ (for any prime number $p$) is isomorphic
to $\mathbb{Z}_{p^{\infty}}$ itself. Thus
$\mathbb{Z}_{p^{\infty}}$ is a coprime $\mathbb{Z}$-module.
Therefore $\mathbb{Z}_{p^{\infty}}\oplus\mathbb{Q}$ is a coprime
$\mathbb{Z}$-module (see {\rm\cite[Lemma 1.3.9]{wij}}). But
$\mathbb{Z}_{p^{\infty}}\oplus\mathbb{Q}$ is not endo-coprime,
because ${\rm
Hom}_{\mathbb{Z}}(\mathbb{Z}_{p^{\infty}},\mathbb{Q})=0$ and thus
${\rm End}_{\mathbb{Z}}(\mathbb{Z}_{p^{\infty}}\oplus\mathbb{Q})$
cannot be a prime ring.}
\end{example}

For any fully invariant submodules $K,L\subseteq M$, consider the
inner coproduct
\begin{eqnarray}
 K:_{M}L &=& \bigcap\{(L)f^{-1}\mid f\in{\rm
End}_{R}(M),K\subseteq {\rm Ker}~f\}   \nonumber \\
   &=& {\rm Ker}(\pi_{K}\diamond{\rm Hom}_{R}(M/K,M)\diamond\pi_{L}), \nonumber
\end{eqnarray}
where $\pi_{K}:M\rightarrow M/K$ and  $\pi_{L}:M\rightarrow M/L$
denote the canonical
projections.\\
By {\rm\cite[Proposition 1.7.3]{wij}}, $M$ is fully coprime if and
only if for any fully invariant submodules $K,L$ of $M$, the
relation $K:_{M}L=M$ implies $K=M$ or $L=M$. Notice that such a
coproduct is considered in Bican et.al. \cite{bic} for any pair of
submodules $K, L\subseteq M$ (not necessary fully invariant) and
then a definition of "coprime modules" is derived from this
coproduct.\\

In \cite{wij}, Wijayanti introduced the following condition for an $R$-module $M$:\\
$(\ast\ast)$ For any proper (fully invariant) submodule $K$ of
$M$, ${\rm Ann}_{R}(K)\nsubseteq{\rm Ann}_{R}(M)$.

\begin{proposition}\label{main}
Let $M$ be an $R$-module with $S={End}_{R}(M)$.\\
$(1)$ If $_{R}M$ is coprime and satisfies
$(\ast\ast)$, then it is endo-coprime\\
$(2)$ If $R$ is commutative and $_{R}M$ is endo-coprime,
then $_{R}M$ is coprime.\\
$(3)$ If $_{R}M$ is endo-coprime and fi-coretractable, then
for every left ideal $I$ of $R$ either $IM=0$ or $IM=M$.\\
$(4)$ If $_{R}M$ is fi-coretractable and $S$ is prime, then $M$ is
fully coprime.
\end{proposition}

\begin{proof}
$(1)$ Let $_{R}M$ be coprime and $K$ be a proper fully invariant
submodule of $M$. If there exists $0\neq f\in {\rm
Ann}_{S}(M/{K})$, then we have
$${\rm Ann}_{R}(M)={\rm Ann}_{R}(M/{{\rm Ker}~f})={\rm
Ann}_{R}((M)f),$$ by coprimeness of $M$. Thus ${\rm
Ann}_{R}(K)\subseteq{\rm Ann}_{R}(M)$, a
contradiction since $M$ satisfies $(\ast\ast)$.\\
$(2)$ Let $K$ be a proper fully invariant submodule of $M$ and
$r\in {\rm Ann}_{R}(M/{K})$. Since $R$ is commutative,
multiplication on $M$ by $r$ (say$~f_{r}$) is indeed an
$R$-endomorphism of $M$. Therefore $(M)f_{r}\subseteq K$ and
endo-coprimeness of $M$ implies that $f_{r}=0$. Consequently
$r\in{\rm Ann}_{R}(M)$ as desired.\\
$(3)$ Suppose that $_{R}M$ is endo-coprime and $fi$-coretractable.
Let $0\neq IM\neq M$ for some left ideal $I$ of $R$. Since $_{R}M$
is $fi$-coretractable, there is $0\neq f\in S$ such that
$(IM)f=0$, consequently $I(M)fS=0$ which implies that
$(M)fS\subsetneq M$. Since $_{R}M$ is $fi$-coretractable, there is
$0\neq g\in S$ such that $fSg=0$, contradicting the primeness of
$S$.\\
$(4)$ For proper fully invariant submodules $K,L$ of $M$, by
$fi$-coretractibility we have $\pi_{K}\diamond{\rm
Hom}_{R}(M/K,M)\neq0$, $\pi_{L}\diamond{\rm Hom}_{R}(M/L,M)\neq0$
and
$$\pi_{K}\diamond{\rm Hom}_{R}(M/K,M)\diamond\pi_{L}\diamond{\rm
Hom}_{R}(M/L,M)\neq0,$$ because $S$ is prime. Thus $K:_{M}L\neq
M$. Consequently $M$ is fully coprime.
\end{proof}

As a consequence we obtain:
\begin{corollary}\label{coprime}
For $_{R}M$ suppose that at least one of the following conditions
hold:\\
$(1)$ $_{R}M$ is $fi$-coretractable and satisfies the $(\ast\ast)$ condition.\\
$(2)$ $_{R}M$ satisfies the $(\ast\ast)$ condition and $R$ is
commutative.\\
Then the following statements are equivalent:

$(a)$ $_{R}M$ is coprime.

$(b)$ $_{R}M$ is endo-coprime.

$(c)$ $_{R}M$ is fully coprime.

$(d)$ For every left ideal $I$ of $R$ either $IM=0$ or $IM=M$.

$(e)$ ${Ann}_{R}(M)$ is a prime ideal of $R$.
\end{corollary}

\begin{proof}
Suppose $_{R}M$ is $fi$-coretractable with $(\ast\ast)$ condition.
Then $(a)\Rightarrow (b)$ by part $(1)$ of Proposition \ref{main}.\\
$(b)\Rightarrow (c)$. Since $S$ is prime, we can apply part (4) of \ref{main}.\\
$(c)\Rightarrow (a)$ is trivial.\\
$(b)\Rightarrow (d)$ by part $(3)$ of Proposition \ref{main}, and
$(d)\Rightarrow (e)$ is clear.\\
$(e)\Rightarrow (a)$ by {\rm\cite[Lemma 1.3.3 part (2)]{wij}}.\\
We now assume $(2)$. It is easy to see that $_{R}M$ is
$fi$-coretractable and so the conditions are equivalent for
$_{R}M$ by the previous case.
\end{proof}
$~~~~~$ Now we offer a short outline of basic facts about
coalgebras and comodules referring to Brzezi\'{n}ski and Wisbauer
\cite{brz} for details, and then we will state Proposition
\ref{main} and Corollary \ref{coprime} for comodules and
coalgebras (as in Jara et. al. \cite{jar}, Nekooei and Torkzadeh \cite{nek}, Wijayanti \cite{wij} and Xu et. al. \cite{xu}).\\

Let $C$ be a coalgebra over a commutative ring $R$ with counit
$\varepsilon:C\rightarrow R$. Its dual algebra is $C^{*}={\rm
Hom}_{R}(C,R)$ with multiplication defined as
$$(f*g)(c)=\mu\circ(f\otimes g)\circ\Delta(c)=\sum f(c_{\underline{1}})g(c_{\underline{2}})$$
for any $c\in C$ and $f,g\in C^{*}$, where $\Delta(c)=\sum
c_{\underline{1}}\otimes c_{\underline{2}}$ (Sweedler's
notation).\\
$~~~~~~$  A right $C$-{\it comodule} is an $R$-module $M$ with an
$R$-linear map $\varrho^{M}:M\rightarrow M\otimes_{R}C$ called a
right $C$-{\it coaction}, with the properties
$$(I_{M}\otimes\Delta)\circ\varrho^{M}=(\varrho^{M}\otimes I_{C})\circ\varrho^{M}~~~~ and~~~~(I_{M}\otimes\varepsilon)\circ\varrho^{M}=I_{M}.$$
$~~~~~$Denote by ${\rm Hom}^{C}(M,N)$ the set of $C$-comodule
morphisms from $M$ to $N$. The class of right comodules over $C$
together with the comodule morphisms form an additive category,
which is denoted by $\mathbf{M}^{C}$.\\
$~~~~~$Similarly to the classical Hom-tensor relations, there are
Hom-tensor relations in $\mathbf{M}^{C}$ (see \cite{brz}). For any
$M\in\mathbf{M}^{C}$ and $X\in\mathbf{M}_{R}$, there is an
$R$-linear isomorphism $$\phi:{\rm
Hom}^{C}(M,X\otimes_{R}C)\rightarrow{\rm
Hom}_{R}(M,X),~~~~f\mapsto(I_{X}\otimes\varepsilon)\circ f,$$ For
$X=R$ and $M=C$ the map $\phi$ yields an algebra
(anti-)isomorphism ${\rm End}^{C}(C)\simeq C^{*}$.\\
$~~~~~$Any $M\in \mathbf{M}^{C}$ is a (unital) left $C^{*}$-module
by $$\rightharpoonup:C^{*}\otimes_{R}M\rightarrow M,~~~~f\otimes
m\mapsto(I_{M}\otimes f)\circ\varrho^{M}(m),$$ and any morphism
$h:M\rightarrow N$ in $\mathbf{M}^{C}$ is a left $C^*$-module
morphism, i.e., $${\rm Hom}^{C}(M,N)\subset{_{C^*}{\rm
Hom}(M,N)}.$$$C$ is a subgenerator in $\mathbf{M}^{C}$, that is
all $C$-comodules are subgenerated by $C$ as $C$-comodules and
$C^{*}$-modules. Thus we have a faithful functor from
$\mathbf{M}^{C}$ to $_{C^*}\mathbf
M$, where the latter denotes the category of left $C^*$-modules.\\
$~~~~~$$C$ satisfies the $\alpha$-{\it condition} if the following
map is injective for every $N\in\mathbf{M}_{R}$,
$$\alpha_{N}:N\otimes_{R}C\rightarrow{\rm Hom}_{R}(C^{*},N),~~~~n\otimes c\mapsto[f\mapsto f(c)n].$$
$~~~~~$$\mathbf{M}^{C}$ is a full subcategory of $_{C^*}\mathbf M$
if and only if $C$ satisfies the $\alpha$-condition, and then
$\mathbf{M}^{C}$ is isomorphic to $\sigma[_{C^*}C]$.\\
$~~~~~${\it Throughout}, $C$ will be an $R$-coalgebra which
satisfies the $\alpha$-condition.\\

For the following definition and more information on the topic,
see \cite{wij}.
\begin{definition}
{\rm Recall that a right $C$-comodule $M$ with $S={\rm End}_{C}(M)$ is\\
--{\it coprime, fully coprime and strongly
coprime} if $M$ is coprime, fully coprime and strongly coprime as a left $C^{*}$-module, respectively.\\
--{\it endo-coprime} if for any proper fully invariant subcomodule
$K$ of $M$, ${\rm Ann}_{S}(M/K)=0$.}
\end{definition}

\begin{proposition}\label{coco}
Let $M$ be a right $C$-comodule such that at least one of the following conditions hold on $M$:\\
$(1)$ $M$ is $fi$-coretractable and satisfies the $(\ast\ast)$ condition as a left $C^*$-module.\\
$(2)$ $M$ satisfies the $(\ast\ast)$ condition and $C^*$ is a commutative algebra.\\
Then the following statements are equivalent:

$(a)$ $M$ is a coprime comodule.

$(b)$ $M$ is a endo-coprime comodule.

$(c)$ $M$ is a fully coprime comodule.

$(d)$ For every left ideal $I$ of $C^{*}$ either $I\rightharpoonup
M=0$ or $I\rightharpoonup M=M$.

$(e)$ ${Ann}_{C^*}(M)$ is a prime ideal of $C^*$.
\end{proposition}

\begin{proof}
By Corollary \ref{coprime}.
\end{proof}

For $M=C$, the assertions in \ref{coco} yield the following
theorem that is a part of the main result 2.11.7 of \cite{wij},
that concluded from {\rm\cite[1.7.11]{wij}}. But {\rm\cite[1.7.11
part (i)]{wij}} has an incorrect proof, because in its proof for
finitely generated right ideals $I,J\subseteq{\rm End}_{R}(M)$,
submodules $K={\rm Ker}~I$ and $L={\rm Ker}~J$ are not fully
invariant. If we consider $I$ and $J$ as finitely generated left
or two sided ideals, then we cannot apply {\rm\cite[1.1.9 part
(2)]{wij}}.
\begin{theorem}
If $C$ is a coalgebra over a field $k$, then the following
statements are equivalent:\\
$(a)$ $C$ is coprime as a right $C$-comodule.\\
$(b)$ $C$ is coprime as a left $C$-comodule.\\
$(c)$ $C$ is endo-coprime as a right $C$-comodule.\\
$(d)$ $C$ is fully coprime as a right $C$-comodule.\\
$(e)$ For every left ideal $I$ of $C^{*}$ either $I\rightharpoonup
C=0$ or $I\rightharpoonup C=C$.\\
$(f)$ $C^*$ is a prime algebra.
\end{theorem}

\begin{proof}
Over a field, $C$ is self-cogenerator and so $fi$-coretractable.
According to {\rm\cite[Lemma 2.2.9]{wij}}, $C$ satisfies condition
$(**)$ as a left $C^*$-module. Also by {\rm\cite[2.6.3 part
(i)]{wij}}, $C$ is coprime as a left $C$-comodule if and only if
$C$ is endo-coprime. On the other hand $C$ is a faithful
$C^{*}$-module (see {\rm\cite[4.6 part (2)]{brz}}).
\end{proof}

The following Proposition shows that a self-injective coprime
module (in the sense of Bican et.al. \cite{bic}) is endo-coprime.
\begin{proposition}
Let $_{R}M$ be self-injective. If every non-zero factor module of
$M$ generates $M$, then ${End}_{R}(M)$ is a prime ring and $_{R}M$
is endo-coprime.
\end{proposition}

\begin{proof}
Let $f\in S:={\rm End}_{R}(M)$ and $I:=fS$. By assumption $M$ is
generated by $M/{\rm Ker}~I$, i.e., there is a short exact
sequence ${(M/{\rm Ker}~I)}^{(\Lambda)}\longrightarrow
M\longrightarrow0$. Applying ${\rm Hom}_{R}(-,M)$ to this exact
sequence yields the commutative diagram

$$\begin{CD}
0 @>>> {\rm Hom}_{R}(M,M) @>>> {\rm Hom}_{R}({(M/{\rm Ker}~I)}^{(\Lambda)},M) \\
@. @VV=V @VV{\simeq}V \\
0 @>>> S @>{\phantom{\text{ very long label}}}>> {I}^{\Lambda},
\end{CD}$$

since ${\rm Hom}_{R}({(M/{\rm Ker}~I)}^{(\Lambda)},M)\simeq{\rm
Hom}_{R}(M/{\rm Ker}~ I,M)^{\Lambda}$ and, by injectivity of $M$,
$\pi_{{\rm Ker}~I}\diamond{\rm Hom}_{R}(M/{\rm Ker}~I,M)=I$
(see{\rm\cite[28.1 part (4)]{wis1}}). Thus $I$ is a faithful right
$S$-module. Consequently ${\rm End}_{R}(M)$ is a prime ring.\\
By assumption, ${\rm Tr}(M/K,M)=M$ for all proper submodule
$K\subset M$. Thus ${\rm Hom}_{R}(M/K,M)\neq 0$, i.e., $M$ is
coretractable. Then by {\rm\cite[1.5.2 part (iii)]{wij}}, $M$ is
endo-coprime.
\end{proof}

Let $N$ be a fully invariant submodule of $M$ and $f\in{\rm
End}_{R}(M)$, then we define $\overline{f}\in{\rm End}_{R}(M/N)$
with $(m+N)\overline{f}=(m)f+N$ (or
$f\pi_{N}=\pi_{N}\overline{f}$, where $\pi_{N}$ is the canonical
projection $M\rightarrow M/N$).\\

We shall need the following Proposition that is a slight
modification of 1.5.3 of \cite{wij}.
\begin{lemma}\label{map}
Let $M$ be an $R$-module and $N$ be a proper fully invariant
submodule of $M$. Consider the ring homomorphism $\psi:{
End}_{R}(M)\rightarrow {End}_{R}(M/N)$ with
$f\mapsto\psi(f):=\overline{f}$, where
$\overline{f}$ defined in the above statement.\\
$(1)$ If $M$ is endo-coprime or $M$ is copolyform and $N\ll M$,
then $\psi$ is injective.\\
$(2)$ If $M$ is self-projective, then $\psi$ is surjective.
\end{lemma}

\begin{proof}
We only need to show that if $M$ is copolyform and $N\ll M$, then
$\psi$ is injective, the other statements are proved in
{\rm\cite[1.5.3]{wij}}. Suppose that $\psi(f)=\overline{f}=0$.
Then $(M)f\subseteq N$, so $(M)f\ll M$. Since $M$ is copolyform,
$(M)f$ is $M$-corational. Consequently $f\in{\rm
Hom}_{R}(M,(M)f)=0$.
\end{proof}

Further properties of endo-coprime modules are collected in the
following.
\begin{proposition}\label{Dedekind}
Let $_{R}M$ be an endo-coprime module. Then the following statements hold:\\
$(1)$ $M$ connot have a nontrivial fully invariant submodule which
is a direct summand. Consequently $R/{Ann}_{R}(M)$ has
no nontrivial central idempotents.\\
$(2)$ If $M=K\oplus L$ for non-zero submodules $K$ and $L$ of $M$,
then ${Hom}_{R}(K,L)\neq0$.\\
$(3)$ If $M$ is coretractable, then any proper fully invariant
submodule $N\subset M$ is superfluous and $M/N$ is
coretractable.\\
$(4)$ For any proper fully invariant submodule $N\subset M$, $M/N$
is endo-coprime
provided that $M$ is self-projective\\
$(5)$ $M$ is Dedekind finite if and only if there exists a proper
fully invariant submodule $N\subset M$ for which $M/N$ is
Dedekind finite.\\
$(6)$ If $M$ is semi-projective, then the center of
$S={End}_{R}(M)$ is a field.
\end{proposition}

\begin{proof}
$(1)$ Since $M$ is endo-coprime, then $S={\rm End}_{R}(M)$ is a
prime ring. Similar to the proof of {\rm\cite[Proposition 1.9 part
(1)]{hag}}, $M$ cannot have a fully invariant submodule which is a
direct summand. Let $e$ be a nontrivial central idempotent element
of $R/{\rm Ann}_{R}(M)$. Thus $M$ can be decomposed to
$M=eM\oplus(1-e)M$. Then $eM=0$ or $(1-e)M=0$. Consequently $e=0$
or $e=1$, since $M$ is a faithful $R/{\rm Ann}_{R}(M)$-module.\\
$(2)$ Assume that $M=K\oplus L$ for non-zero submodules $K$ and
$L$ of $M$ such that ${\rm Hom}_{R}(K,L)=0$. Then ${\rm
End}_{R}(M)=\left[\begin{array}{cccc}{\rm End}_{R}(K) & {\rm
Hom}_{R}(L,K)\\0 & {\rm End}_{R}(L)\end{array}\right]$. It follows
that $[0~~~L]{\rm End}_{R}(M)\subseteq[0~~~L]$. Thus $(0\oplus L)$
is a fully invariant submodule of $M$. By part (1), $K=0$ or $L=0$, a contradiction.\\
$(3)$ Suppose $N$ is not superfluous in $M$ and $N+K=M$ for some
proper submodule $K$. Since $M$ is coretractable, there exists a
non-zero $f\in S={\rm End}_{R}(M)$ such that $(K)f=0$. Then
$(M)f\subseteq N$, which contradicts the assumption that $M$ is
endo-coprime. Now to show that $M/N$ is coretractable, let $K/N$
be a proper submodule of $M/N$. Then there exists $0\neq f\in
S={\rm End}_{R}(M)$ with $(K)f=0$. Now define the map $g\in{\rm
End}_{R}(M/N)$ by $(m+N)g=(m)f+N$. If $g=0$, then $(M)f\subseteq
N$, which contradicts the assumption that $M$ is endo-coprime.
On the other hand $(K/N)g=0$. Consequently $M/N$ is coretractable.\\
$(4)$ Let $K/N$ be a proper fully invariant submodule of $M/N$.
Then by {\rm\cite[Corollary 1.1.21]{wij}}, $K$ is a proper fully
invariant submodule of $M$. Suppose there exists
$\overline{f}\in{\rm End}_{R}(M/N)$ with
$(M/N)\overline{f}\subseteq{K/N}$. By Lemma \ref{map} there is
$f\in{\rm End}_{R}(M)$ with $f\pi_{N}=\pi_{N}\overline{f}$. Thus
$(M)f\subseteq K$, which implies $f=0$, because $M$ is
endo-coprime. Consequently
$\overline{f}=0$.\\
$(5)$ The necessity is clear. Conversely, let $N\subset M$ be a
proper fully invariant submodule which $M/N$ is Dedekind finite.
Since $M$ is endo-coprime, ${\rm End}_{R}(M)$ is isomorphic to a
subring of ${\rm End}_{R}(M/N)$, by Lemma \ref{map}. The remainder
of proof is similar to {\rm\cite[Proposition 1.9 part (4)]{hag}}.\\
$(6)$ Suppose that $f$ is a non-zero central element of $S$. Then
$(M)f$ is a fully invariant submodule of $M$. If $(M)f\neq M$,
then $f\in{\rm Ann}_{S}(M/(M)f)=0$, which contradicts our
assumption. Thus $(M)f=M$, and because $M$ is semi-projective, we
have $S=Sf$ which means that $f$ is invertible.
\end{proof}

\begin{corollary}
Let $_{R}M$ be self-projective copolyform with $N$ a superfluous
fully invariant submodule. Then $_{R}M$ is endo-coprime if and
only if so is $M/N$.
\end{corollary}

\begin{proof}
$(\Rightarrow)$. By part $(3)$ of Proposition \ref{Dedekind}.\\
$(\Leftarrow)$. Let $M/N$ be endo-coprime and $K$ be a proper
fully invariant submodule of $M$, with $(M)f\subseteq K$ for some
$f\in{\rm End}_{R}(M)$. Then $K+N\neq M$, and $(K+N)/N$ is fully
invariant in $M/N$ by {\rm\cite[Lemma 1.1.20 part (ii)]{wij}}.  We
have $(M)f\subseteq K+N$, hence
$(M/N)\overline{f}\subseteq(K+N)/N$ for $\overline{f}$ which
defined in prior of Lemma \ref{map}. Since $M/N$ is an
endo-coprime module we must have $\overline{f}=0$. But then again
by Lemma \ref{map} we deduce that $f=0$.
\end{proof}

\begin{remark}{\rm
A direct sum of endo-coprime modules need not be endo-coprime. To
see this we recall that the endomorphism rings of the quasi-cyclic
group ${\mathbb Z}_{p^{\infty}}$ and the group of $p$-adic
integers ${{\mathbb Q}_{p}}^*$ are isomorphic commutative domains.
On the other hand ${\mathbb Z}_{p^{\infty}}$ is a self-cogenerator
${\mathbb Z}$-module, so by {\rm\cite[1.5.2 part (iii)]{wij}},
${\mathbb Z}_{p^{\infty}}$ is endo-coprime. But ${\rm
End}_{\mathbb Z}({\mathbb Q}/{\mathbb
Z})\simeq\prod_{p}\mathbb{Q}_{p}^*$ is not a prime ring, thus
${\mathbb Q}/{\mathbb Z}$ is not an endo-coprime $\mathbb
Z$-module.}
\end{remark}

A non-zero module is called {\it endo-simple} if it has no
nontrivial fully invariant submodules.
\begin{lemma}\label{sc}
Let $_{R}M$ be a projective module in $\sigma[M]$. Then the
following statements are equivalent:\\
$(a)$ $M$ is fully coprime.\\
$(b)$ $M$ is strongly coprime.\\
$(c)$ $M$ is endo-simple.
\end{lemma}

\begin{proof}
$(a)\Leftrightarrow(c)$. See {\rm\cite[Corollary 4.6]{rag}}, but
notice that two concepts coprime modules in \cite{rag}
and fully coprime modules in \cite{wij} are coincide.\\
$(b)\Leftrightarrow(c)$. Assume that $M$ is projective in
$\sigma[M]$. Then for any non-zero fully invariant submodule $U$
of $M$, $M\notin\sigma[M/U]$ (see {\rm\cite[Lemma 2.8]{vio}}). In
this case if $M$ is also strongly coprime, then $M$ has no
nontrivial fully invariant submodules.
\end{proof}

The following result generalizes {\rm\cite[1.7.4]{wij}}.
\begin{corollary}
The following statements are equivalent for the ring $R$:\\
$(a)$ $_{R}R$ is coprime.\\
$(b)$ $_{R}R$ is endo-coprime.\\
$(c)$ $_{R}R$ is  fully coprime.\\
$(d)$ $_{R}R$ is strongly coprime.\\
$(e)$ $R$ is a simple ring.
\end{corollary}

\begin{corollary}
Let $_{R}M$ be projective in $\sigma[M]$. If $M$ is strongly
coprime or fully coprime, then $M$ is copolyform.
\end{corollary}

\begin{proof}
By Lemma \ref{sc} and {\rm\cite[1.9.15]{wij}}.
\end{proof}

Endo-coprimeness is preserved under isomorphism. We further have:
\begin{proposition}
Being endo-coprime is a Morita invariant property.
\end{proposition}

\begin{proof}
Assume that $A$ and $B$ are Morita equivalent rings with inverse
category equivalences $\alpha:A$-${\rm Mod}\rightarrow B$-${\rm
Mod}$ and $\beta:B$-${\rm Mod}\rightarrow A$-${\rm Mod}$. Let $M$
be an endo-coprime $A$-module, we want to show that $(M)\alpha$ is
an endo-coprime $B$-module. Suppose that $N$ is a proper fully
invariant submodule of $(M)\alpha$ and $h\in{\rm
End}_{B}((M)\alpha)$ such that $((M)\alpha)h\subseteq N$. Let $i$
denotes the inclusion map from $N$ to $(M)\alpha$, $g:=(i)\beta$
and $N':=(N)\beta$. We have
$(M)\alpha\beta(h)\beta\subseteq(N')g$. Since $(N')g$ is a proper
fully invariant submodule of $(M)\alpha\beta$, by the
endo-coprimeness of the latter, we deduce that $(h)\beta=0$, and
consequently $h=0$. This shows that $(M)\alpha$ is an endo-coprime
$B$-module.
\end{proof}

The following proposition is a generalization of
{\rm\cite[1.15]{hag}}.
\begin{proposition}\label{semisimple}
The following statements are equivalent on a module $_{R}M$:\\
$(a)$ $_{R}M$ is homogeneous semisimple.\\
$(b)$ Each non-zero element of $\sigma[M]$ is an endo-prime
module.\\
$(c)$ Each non-zero element of $\sigma[M]$ is an endo-coprime
module.\\
$(d)$ Each non-zero element of $\sigma[M]$ is an endo-simple
module.
\end{proposition}

\begin{proof}
$(a)\Leftrightarrow(b)$. {\rm\cite[Proposition 1.15]{hag}}.\\
$(d)\Rightarrow(b)$ and $(d)\Rightarrow(c)$ are obvious.\\
$(b)\Rightarrow(d)$. Let $N\in\sigma[M]$ with $S={\rm End}_{R}(N)$
and $K$ be a nontrivial fully invariant submodule of $N$. Since
$(N/K)\bigoplus N$ is endo-prime, ${\rm End}_{R}((N/K)\bigoplus
N)$ is a prime ring and so ${\rm Ann}_{S}(K)=\pi_{K}\diamond{\rm
Hom}_{R}(N/K,N)\neq0$,
that contradicts endo-primeness of $N$.\\
$(c)\Rightarrow(d)$. Let $N\in\sigma[M]$ with $S={\rm End}_{R}(N)$
and $K$ be a nontrivial fully invariant submodule of $N$. Since
${\rm End}_{R}(N\bigoplus K)$ is a prime ring, $ {\rm
Ann}_{S}(N/K)={\rm Hom}_{R}(N,K)\neq0$, which is a contradiction
with our hypothesis.
\end{proof}

\begin{proposition}\label{sa}
The following statements are equivalent on a ring $R$:\\
$(a)$ $R$ is a simple Artinian ring.\\
$(b)$ Each non-zero left $R$-module is endo-prime.\\
$(c)$ Each non-zero finitely generated left $R$-module is endo-prime.\\
$(d)$ Each non-zero left $R$-module is endo-coprime.\\
$(e)$ Each non-zero finitely generated left $R$-module is endo-coprime.\\
$(f)$ Each non-zero left $R$-module is endo-simple.
\end{proposition}

\begin{proof}
By \ref{semisimple} the statements $(a)$, $(b)$, $(d)$ and $(f)$
are equivalent. \\
$(b)\Rightarrow(c)$ and $(d)\Rightarrow(e)$ are trivial.\\
$(c)\Rightarrow(a)$. See {\rm\cite[Proposition 1.16]{hag}}.\\
$(e)\Rightarrow(a)$. Clearly $R$ must be prime ring. For any
simple module $_{R}N$ and proper left ideal $I$ of $R$, since
$N\bigoplus R$ and $R/I\bigoplus R$ are endo-coprime, their
endomorphism rings are prime and consequently ${\rm Hom}_{R}(N,R)$
and ${\rm Hom}_{R}(R/I,R)$ cannot be zero. The reminder of the
proof is similar to {\rm\cite[Proposition 1.16]{hag}}.
\end{proof}

\section{Duality for Endo-Prime and Endo-Coprime Modules}

$~~~~~$In this section, we shall investigate $R$-duals. If $M$ is
an $R$-module, then the $R$-dual of $M$ is ${\rm Hom}_{R}(M,R)$,
and this will be denoted by $M^*$. Notice that $M^*$ is a right
$R$-module and so its endomorphism ring ${\rm End}_{R}(M^{*})$
will act on the left side of $M^{*}$. The double dual of $M$, that
is, $({M^*})^*$, will be denoted by $M^{**}$. If the natural map
$\iota:M\rightarrow M^{**}$ given by
$$[(m)\iota](f)=(m)f~~~(m\in M, f\in M^*),$$ is injective, $M$
will be called {\it torsionless}. It is clear that $M$ is
torsionless if and only if the reject ${\rm Rej}(M,R)$ is zero. A
torsionless module $_{R}M$ is said to be {\it reflexive} if the
injection map $\iota$ is in fact an isomorphism (of left
R-modules). For any submodule $K\subseteq M$, put
$K^{\bot}:=\{f\in M^{*}\mid(K)f=0\}$, which is a submodule of
$M^{*}$. Similarly, for any submodule $I\subseteq M^{*}$, put
$I^{\bot}:=\bigcap\{{\rm Ker}~f\mid f\in I\}$ (a submodule of
$M$). We recall that, for any $f\in{\rm End}_{R}(M)$, the map
$f^{*}:M^{*}\rightarrow M^{*}$ given by $f^{*}(g)=fg$ is a right
$R$-module homomorphism. If $M$ is a reflexive R-module, then the
mapping $()^{*}:{\rm End}_{R}(M)\rightarrow{\rm End}_{R}(M^{*})$
is a ring anti-isomorphism (see {\rm\cite[Proposition 2.2]{cor}}).
We recall from {\cite[Theorem 15.11]{lam}} that, over
quasi-Frobenius (QF) ring $R$ any finitely generated module
$_{R}M$ is reflexive. Furthermore, under this assumption, for any
submodules $K\subseteq M$ and $I\subseteq M^{*}$ we have
$K^{\bot\bot}=K$ and $I^{\bot\bot}=I$. The reader is referred to
Lam \cite{lam} and Wisbauer \cite{wis1}.

\begin{theorem}
Let $M$ be a finitely generated left module over $QF$ ring $R$. Then the following statements hold:\\
$(1)$ $_{R}M$ is endo-prime if and only if $M^{*}$ is endo-coprime (as a right $R$-module).\\
$(2)$ $_{R}M$ is endo-coprime if and only if $M^{*}$ is endo-prime
(as a right $R$-module).
\end{theorem}
\begin{proof}
$(1)$ $(\Rightarrow)$. Let $I$ be a proper fully invariant
submodule of $M^*$. If $I^{\bot}=0$, then $I=I^{\bot\bot}=M^{*}$,
which contradicts our assumption. So $I^{\bot}$ is a non-zero
fully invariant submodule of $M$. Let $f^{*}\in{\rm
End}_{R}(M^{*})$ with $f^{*}(M^{*})\subseteq I$. Then for any
$g\in M^{*}$, we have $fg=f^{*}(g)\in I$. Hence
$(I^{\bot})f\subseteq {\rm Rej}(M,R)=0$. Consequently from
endo-primeness of $M$ we deduce that $f=0$ and
so $f^{*}=0$.\\
$(\Leftarrow)$. Let $K$ be a non-zero fully invariant submodule of
$M$. Then for any $f^{*}\in{\rm End}_{R}(M^*)$ and any $g\in
K^{\bot}$ we have $(K)f^{*}(g)=(K)fg\subseteq(K)g=0$. Thus
$K^{\bot}$ is a proper fully invariant submodule of $M^*$. Now let
$h\in{\rm End}_{R}(M)$ with $(K)h=0$. Then for any $g\in M^{*}$,
$(K)h^{*}(g)=(K)hg=0$, i.e., $h^{*}(M^{*})\subseteq K^{\bot}$.
Consequently endo-coprimeness of $M^{*}$ implies that $h^{*}=0$.
Thus ${\rm Im}~h\subseteq{\rm Rej}(M,R)=0$.\\
$(2)$ $(\Rightarrow)$. Let $I$ be a non-zero fully invariant
submodule of $M^*$ and $f^{*}\in{\rm End}_{R}(M^*)$ such that
$f^{*}(I)=0$. Then for any $g\in I$, $(M)f\subseteq{\rm Ker}~g$,
i.e., $(M)f\subseteq I^{\bot }$. Thus $f=0$, because $M$ is
endo-coprime and $I^{\bot}$ is a proper fully invariant submodule of $M$.\\
$(\Leftarrow)$. Now assume that $M^{*}$ is endo-prime and $K$ is a
proper fully invariant submodule of $M$. If $K^{\bot}=0$, then
$K=K^{\bot\bot}=M$, a contradiction. Hence $K^{\bot}$ is a
non-zero fully invariant submodule of $M^{*}$. Let $f\in{\rm
End}_{R}(M)$ with $(M)f\subseteq K$, then for any $g\in K^{\bot}$,
$(M)f^{*}(g)=(M)fg\subseteq(K)g=0$. Therefore $f^{*}(K^{\bot})=0$
and so $f^{*}=0$, because $M^{*}$ is endo-prime. Since $M$ is
torsionless $f$ must be zero.
\end{proof}


\end{document}